\documentclass[a4paper,reqno,12pt]{amsart}
\usepackage{amsmath,amssymb,dsfont}
\usepackage{graphicx}
\newtheorem{theorem}{Theorem}[section]
\newtheorem{proposition}[theorem]{Proposition}
\newtheorem{lemma}[theorem]{Lemma}

\theoremstyle{remark}
\newtheorem{remark}{Remark}[section]
\newcommand{\df}{\mathrm{d}}
\newcommand{\im}{\mathrm{i}}

\newcommand{\Q}{\mathds{Q}}

\begin{document}

\title[On the existence of orthogonal polynomials]{On the existence of orthogonal polynomials for oscillatory weights on a bounded interval}
\author[H. Majidian]{Hassan Majidian}
\address{Department of Basic Sciences, Iranian Institute for Encyclopedia Research, PO Box 14655-478, Tehran, Iran.}
\email{majidian@iecf.ir}
\maketitle
\begin{abstract}
It is shown that the orthogonal polynomials, corresponding to the oscillatory weight $e^{\im\omega x}$, exists if $\omega$ is a transcendental number and $\tan\omega/\omega\in\Q$. Also, it is proved that such orthogonal polynomials exist for almost every $\omega>0$, and the roots are all simple if $\omega>0$ is either small enough or large enough.
\end{abstract}


MSC 65D32; 65R10

\section{Introduction}
We consider the problem of existence of orthogonal polynomials and Gaussian quadrature rules (in the standard form) for the following inner product:
\begin{equation}\label{eq:inner}
(f,g)_{\omega}=\int_{-1}^1 f(x)g(x)e^{\im\omega x}\,\df x,
\end{equation}
with $\omega>0$. More precisely, we seek a monic polynomial $p^{\omega}_n$ of a given degree $n$ such that
\begin{equation}
\int_{-1}^1 p^{\omega}_n(x)x^j e^{\im\omega x}\,\df x=0,\qquad j=0,1,\ldots,n-1.
\end{equation}

The following results on the existence of $p^{\omega}_n$ are due to~\cite{ash12}:
\begin{enumerate}
  \item[]
  Proposition 1: $p^{\omega}_1$ exists for any $\omega$ except when $\omega$ is a multiple of $\pi$;
  \item[]
  Proposition 2: $p^{\omega}_2$ exists for all $\omega$;
  \item[]
  Conjecture~1: $p^{\omega}_n$  with $n$ even exists for all $\omega$;
  \item[]
  Conjecture~2: $p^{\omega}_n$  with $n$ odd does not exists for some $\omega$.
\end{enumerate}

In this paper, we give a sufficient condition on $\omega$ for which $p^{\omega}_n$ exists for all $n$. According to Conjecture~1, this condition is not necessary. We show that $p^{\omega}_n$ exists for almost every $\omega>0$. If the existence of $p^{\omega}_n$ is assumed, it is shown that all of its roots are simple when $\omega>0$ is either small enough or large enough.

Throughout the paper, we frequently suppress the dependence of objects on $\omega$ for simplification in notations.
\section{Orthogonal polynomials}
A necessary and sufficient condition for existence of the orthogonal polynomial $p^{\omega}_n$ is that the Hankel determinant
\begin{equation}
\Delta_n=\left|
           \begin{array}{cccc}
             \mu_0 & \mu_1 & \cdots & \mu_{n-1} \\
             \mu_1 & \mu_2 & \cdots & \mu_n \\
             \vdots & \vdots & \cdots & \vdots \\
             \mu_{n-1} & \mu_n & \cdots & \mu_{2n-2} \\
           \end{array}
         \right|
\end{equation}
does not vanish. The moment $\mu_k:=\int_{-1}^1 x^k e^{\im\omega x}\,\df x$ is defined recursively (see~\cite{ash12}):
\begin{subequations}\label{eq:muk}
\begin{align}
\mu_0&=\frac{2\sin\omega}{\omega},\\[1ex]
\mu_k&=\frac{1}{\im\omega}\left(e^{\im\omega}-(-1)^k e^{-\im\omega}\right)-\frac{k}{\im\omega}\mu_{k-1},\quad k\ge1.
\end{align}
\end{subequations}

It is easy to show that
\begin{equation}\label{eq:mu}
\mu_k=\frac{(-1)^k k!}{(\im\omega)^k}\sum_{\nu=0}^k \frac{(-\im\omega)^{\nu}s_{\nu}}{\nu!},
\end{equation}
where
$$
s_{\nu}:=\frac{1}{\im\omega}\left(e^{\im\omega}-(-1)^{\nu} e^{-\im\omega}\right)=\left\{
\begin{array}{ll}
\dfrac{2\sin\omega}{\omega},&\mbox{for $\nu$ even},\\[1em]
\dfrac{2\cos\omega}{\im\omega},&\mbox{for $\nu$ odd}.
\end{array}\right.
$$
Then we can expand~(\ref{eq:mu}) into
\begin{equation}\label{eq:moment}
\mu_k=\frac{2(-1)^{k+1} k!}{(\im\omega)^k}\left(\cos\omega\sum_{\nu=1\atop \nu\mbox{ \tiny odd}}^k \frac{(-\im\omega)^{\nu-1}}{\nu!}-
\frac{\sin\omega}{\omega}\left(1+\sum_{\nu=2\atop \nu\mbox{ \tiny even}}^k \frac{(-\im\omega)^{\nu}}{\nu!}\right)\right).
\end{equation}

Now consider the matrix corresponding to the Hankel determinant $\Delta_n$. If we take from the $r$th row the factor $\left(\frac{-1}{\im\omega}\right)^{r-1}$, and from the $s$th column the factor $\left(\frac{-1}{\im\omega}\right)^{s-1}$, then we arrive at a new Hankel determinant $\widetilde{\Delta}_n$ with the moments
\begin{equation}\label{eq:tmoment}
\widetilde{\mu}_k:=-2k!\left(\cos\omega\sum_{\nu=1\atop \nu\mbox{ \tiny odd}}^k \frac{(-\im\omega)^{\nu-1}}{\nu!}-
\frac{\sin\omega}{\omega}\left(1+\sum_{\nu=2\atop \nu\mbox{ \tiny even}}^k \frac{(-\im\omega)^{\nu}}{\nu!}\right)\right).
\end{equation}

The relation between $\Delta_n$ and $\widetilde{\Delta}_n$ is then
$$
\Delta_n=\left(\frac{1}{\im\omega}\right)^{n(n-1)}\widetilde{\Delta}_n.
$$
Thus, $\widetilde{\Delta}_n\ne0$ if and only if $\Delta_n\ne0$. If $\omega$ is such that each $\widetilde{\mu}_k$ is a polynomial in $\im\omega$ with rational coefficients, then $\widetilde{\Delta}_n$ is a polynomial in $\im\omega$ with rational coefficients. As the proof of Theorem 2.3 in~\cite{mil05}, we employ the fact that transcendental numbers can not be zeros of a polynomial with rational coefficients. Then we seek a set $S$ of transcendental $\omega$, for which $\widetilde{\mu}_k$ is a polynomial in $\im\omega$ with rational coefficients. Clearly, any multiplier of $\pi$ falls in $S$.

If $\omega\in S$, then $\cos\omega\ne0$. Then the moments can be rewritten as
\begin{equation}
\mu_k=\frac{2(-1)^{k+1} k!\cos\omega}{(\im\omega)^k}\left(\sum_{\nu=1\atop \nu\mbox{ \tiny odd}}^k \frac{(-\im\omega)^{\nu-1}}{\nu!}-
\frac{\tan\omega}{\omega}\left(1+\sum_{\nu=2\atop \nu\mbox{ \tiny even}}^k \frac{(-\im\omega)^{\nu}}{\nu!}\right)\right).
\end{equation}
Again using the above idea, it is enough to determine $\omega>0$ not belonging to $\Q$ (the field of rational numbers) for which
\begin{equation}
\widehat{\mu}_k:=-2k!\left(\sum_{\nu=1\atop \nu\mbox{ \tiny odd}}^k \frac{(-\im\omega)^{\nu-1}}{\nu!}-
\frac{\tan\omega}{\omega}\left(1+\sum_{\nu=2\atop \nu\mbox{ \tiny even}}^k \frac{(-\im\omega)^{\nu}}{\nu!}\right)\right)
\end{equation}
is a polynomial in $\im\omega$. Thus, the problem is to find transcendental numbers $\omega>0$ not belonging to $\{m\pi : m=1,2,\ldots\}$, such that $\tan\omega/\omega\in\Q$.


Transcendental numbers can be zeros of a polynomial with rational coefficients if and only if the polynomial is identically zero. Thus it is enough to show that $\widetilde{\Delta}_n$, $\widehat{\Delta}_n$, as functions of $\im\omega$, are not identically zero for $n>1$. This can be shown by a discussion similar to  that carried out in the proof of Theorem 2.3 in~\cite{mil05}.
Thus we have the following result.
\begin{proposition}\label{prop:main}
For any transcendental $\omega>0$ with $\tan\omega/\omega\in\Q$, the orthogonal polynomial $p^{\omega}_n$ exists.
\end{proposition}
\begin{remark}
The converse is not necessarily true. There are examples of $\omega$ with $\tan\omega/\omega\not\in\Q$ while $\Delta_n\ne0$, i.e., the orthogonal polynomial $p^{\omega}_n$ exists. For example, $p^{\omega}_2$ exists for any $\omega>0$~\cite{ash12}.
\end{remark}

The set $S$ determined in Proposition~\ref{prop:main} is at most countable due to countability of $\Q$. However, our numerical experiences show that $p^{\omega}_n$ exists for almost every $\omega>0$. In the following, we establish this result.
\begin{theorem}
$p^{\omega}_n$ exists for almost every $\omega>0$.
\end{theorem}
\begin{proof}
By induction on the index $k$, we can show from~(\ref{eq:muk}) that the moments $\mu_k$, as functions of $\omega$, are analytic in $D$, an arbitrary connected neighborhood of the semi-axis $\omega>0$. The same result holds then for the Hankel determinant $\Delta_n=\Delta_n(\omega)$. Since zeros of any analytic function (if it is not identically zero) are isolated, it is enough to show that $\Delta_n(\omega)$ is not identically zero in $D$. Since $\Delta_n$ is analytic and then continuous, it is enough to show that $\Delta_n(0)\ne0$; and his can be done similar to the proof of Theorem 2.3 in~\cite{mil05}.
\end{proof}
\section{Gaussian quadrature rules}
Since the weight function in~(\ref{eq:inner}) is not positive, we can not readily claim that the roots of $p^{\omega}_n$ (if exists) are all simple. If $p^{\omega}_n$ have some multiple zeros, then the $n$-point Gaussian quadrature rule can be written in the following form:
$$
G_n(g)=\sum_{\nu=1}^n\sum_{k=0}^{m_{\nu}-1}w_{\nu,k} f^{(k)}\left(x_{\nu}\right),
$$
where $m_{\nu}$ is the multiplicity of the node $x_{\nu}$, and the weights $w_{\nu,k}$ are such that the rule is exact if $f$ is replaced by a polynomial of degree at most $2n-1$. Here in the notations, we suppressed the dependence of the nodes and the weights on $n$. This rule, however, is rarely of practical interest since determining the multiplicities of the nodes is not an easy task. Our numerical experiences show that the roots of $p^{\omega}_n$ (if exists) are all simple.

This result can be established if we assume the existence of $p^{\omega}_n$ for all $\omega>0$. According to Conjecture~2, this result most probably hlods for $n$ even. From our numerical experiences, the same result can be drawn too. We have computed the absolute values of the Hankel determinant for $n=2,4,6$; for each $n$, the graph has been drawn for some increasing $\omega$ (see Figure~\ref{fig:hdeven}). As it is seen, the graphs never cut the horizontal axis, i.e., the Hankel determinants never vanish.
\begin{figure}
\centerline{\fbox{\includegraphics[width=0.3\textwidth,height=12ex]{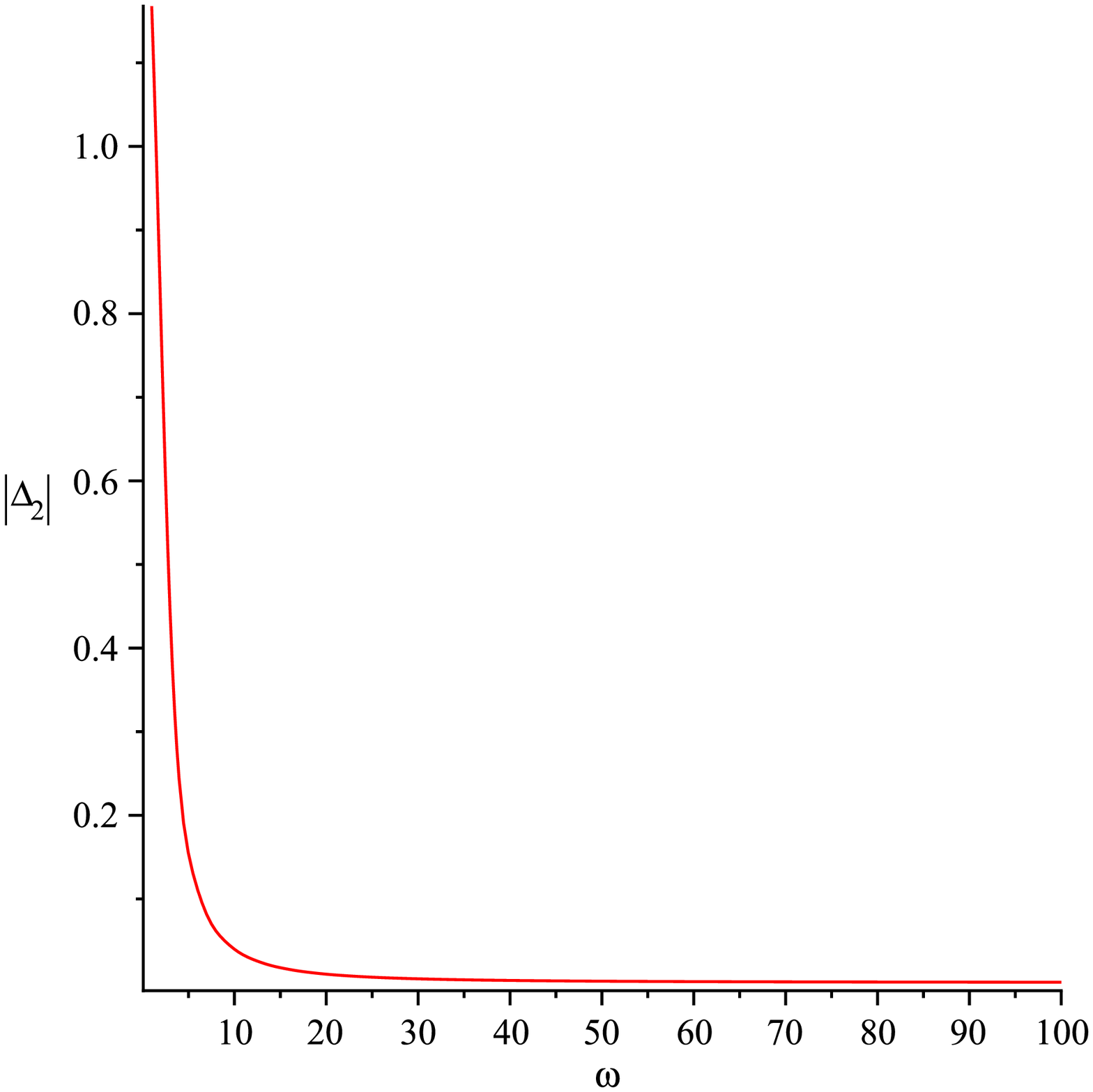}}\quad\fbox{\includegraphics[width=0.3\textwidth,height=12ex]{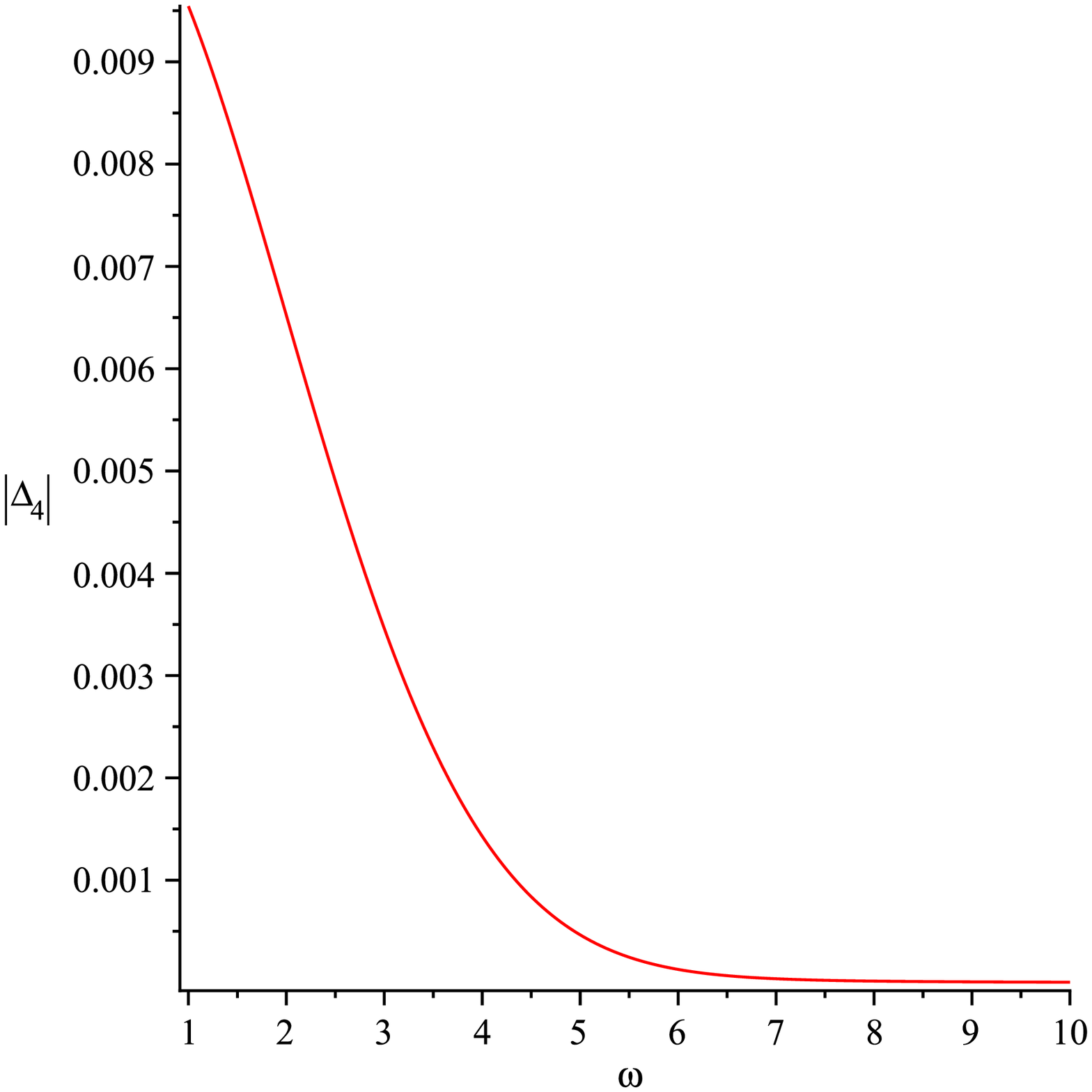}}}\vspace{1ex}
\centerline{\fbox{\includegraphics[width=0.3\textwidth,height=12ex]{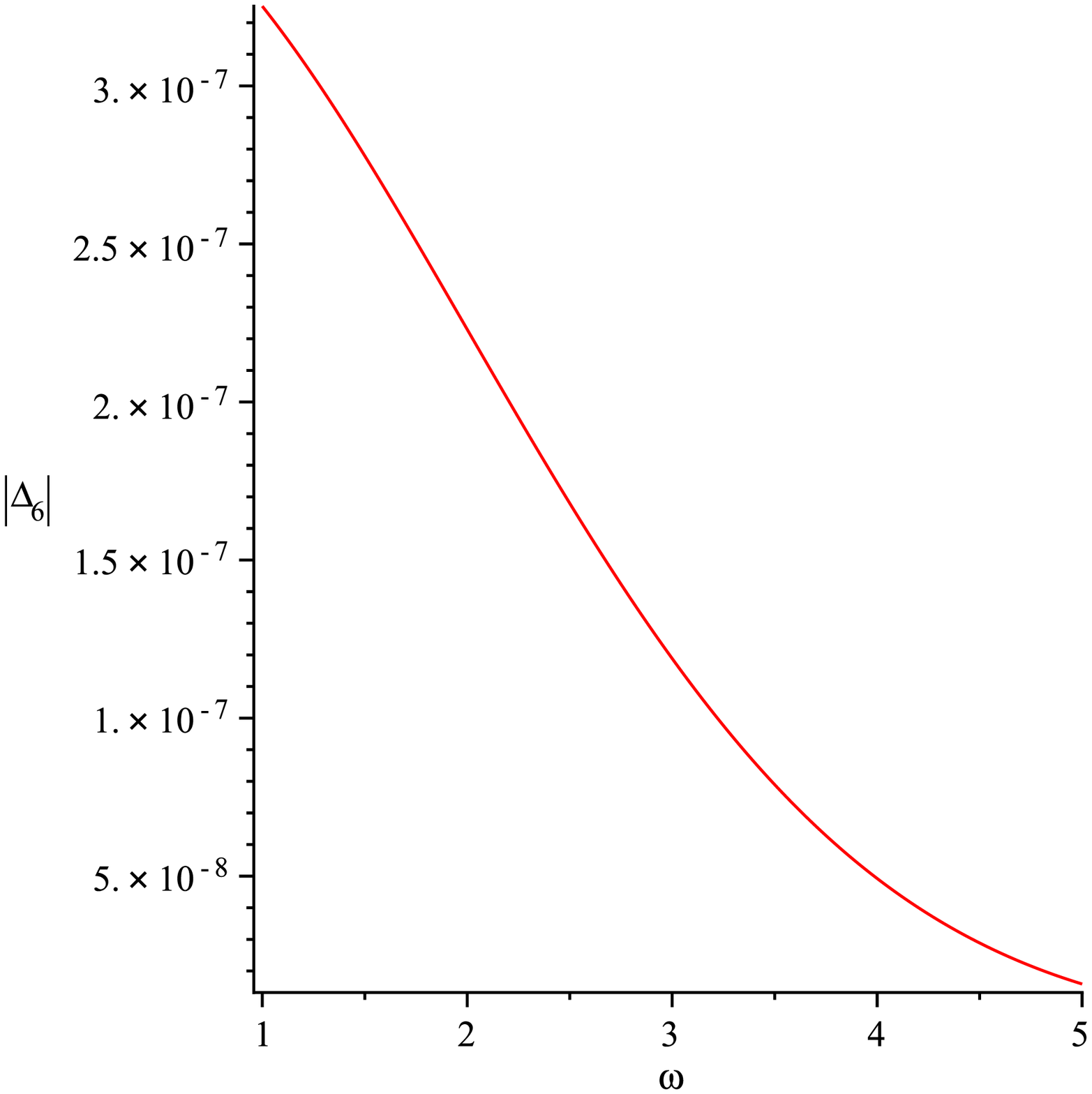}\quad\includegraphics[width=0.3\textwidth,height=12ex]{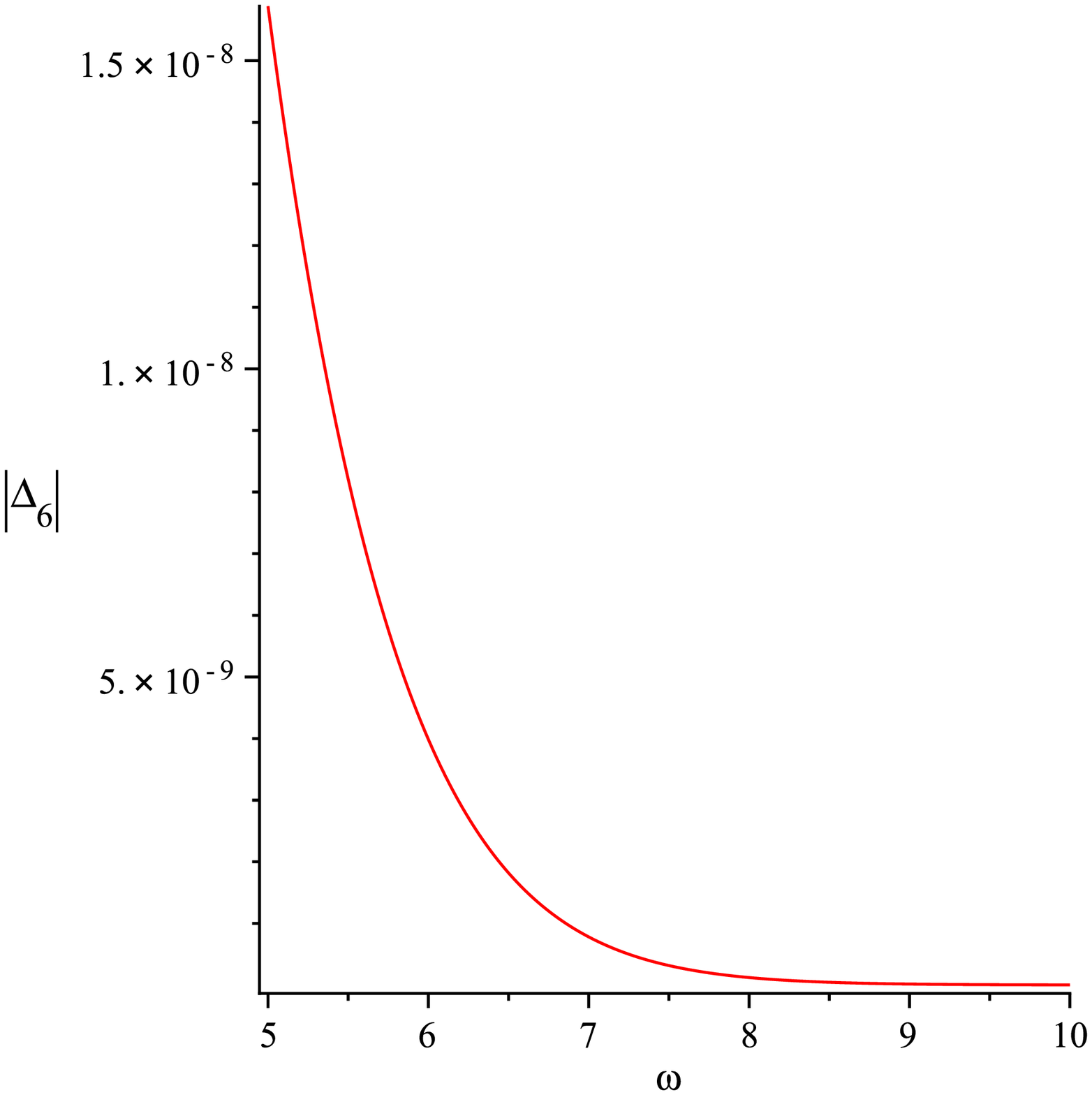}}}
\caption{\small Absolute values of the Hankel determinants $\Delta_n$ for n=2,4,6.\label{fig:hdeven}}
\end{figure}

\begin{lemma}\label{lem:contus}
For a given integer $n>0$, assume that the orthogonal polynomial $p_n^{\omega}(x)$ exists for all $\omega>0$. Then all coefficients of $p_n^{\omega}(x)$ as functions of $\omega$ are continuous.
\end{lemma}
\begin{proof} 
If $p_n^{\omega}(x)=x^n+\sum_{k=0}^{n-1}a_k(\omega)x^k$, then the coefficients $a_0(\omega),\ldots,a_{n-1}(\omega)$ satisfy the linear system
\begin{equation}
\left[v_0(\omega),\ldots,v_{n-1}(\omega)\right] u_n(\omega)+v_n(\omega)=0,
\end{equation}
where
$$
v_k(\omega)=\left[\mu_k,\mu_{k+1},\ldots,\mu_{k+n-1}\right]^\textsf{T},\quad u_n(\omega)=\left[a_0(\omega),\ldots,a_{n-1}(\omega)\right]^\textsf{T}.
$$
Then
\begin{equation}\label{eq:u}
u_n(\omega)=-\frac{1}{\Delta_n}\left[\textsf{V}_n(\omega)\right]^\textsf{T}v_n(\omega),
\end{equation}
where $\textsf{V}_n(\omega)$ is the cofactor matrix of $\left[v_0(\omega),\ldots,v_{n-1}(\omega)\right]$. All entries of the matrix $\textsf{V}_n(\omega)$ are continuous with respect to $\omega$ due to the continuity of the moments $\mu_k$, the entries of $\left[v_0(\omega),\ldots,v_{n-1}(\omega)\right]$. Since the denumerator $\Delta_n$ does not vanishes for any $\omega>0$, the result follows from (\ref{eq:u}). 
\end{proof}
\begin{theorem}
For a given integer $n>0$, assume that the orthogonal polynomial $p_n^{\omega}(x)$ exists for all $\omega>0$. If $\omega>0$ is small enough or large enough, then all of the roots of the orthogonal polynomial $p_n^{\omega}(x)$ are simple.
\end{theorem}
\begin{proof}
It is well-known that the roots of a polynomial vary continuously as the coefficients of the polynomial change continuously. Thus, Lemma~\ref{lem:contus} implies that the trajectories of the roots of $p_n^{\omega}(x)$, as $\omega>0$ increases, are all continuous. Since the roots corresponding to $\omega=0$ as well as $\omega\to\infty$ are all distinct~\cite{ash12}, then the result follows.
\end{proof}
\section{Concluding remarks}
We have shown that the orthogonal polynomial $p^{\omega}_n$, corresponding to the oscillatory weight $e^{\im\omega x}$, exists if $\omega$ is a transcendental number and $\tan\omega/\omega\in\Q$. The set of such $\omega$ is nonempty since it contains the multipliers of $\pi$. Determining other members is not an easy task, so the main problem is still unsolved: For which values of $\omega$ does $p^{\omega}_n$ exist?
We have also shown that $p^{\omega}_n$ exist for almost every $\omega$.

In order to arrive at an $n$-point Gaussian quadrature rule of standard form, it is necessary that all the roots of $p^{\omega}_n$ (if exists) to be simple. The simplicity of the roots of $p^{\omega}_n$ is established only when $\omega>0$ is small enough or when it is large enough. The problem is unsolved for an arbitrary $\omega>0$. We believe that the more properties of $p^{\omega}_n$ one knows, the higher chance he has to solve the problem. For instance, the symmetricity of $p^{\omega}_n$ (cf.~\cite{ash12}) implies that the coefficients of $p^{\omega}_n(z)$ (starting from 1, the coefficient of $z^n$) are real and pure imaginary, alternatively. Also form the three-term recurrence relation,
\begin{equation}\label{eq:3term}
p^{\omega}_k(z)=(z-\alpha_{k-1})p^{\omega}_{k-1}(z)-\beta_{k-1}p^{\omega}_{k-2}(z),
\end{equation}
and Theorem~3.3 of~\cite{ash12}, it is easy to show that $\alpha_k$ and $\alpha'_k$ are pure imaginary numbers; $\beta_k$ and $\beta'_k$ are real. Here the prime sign indicates the derivative with respect to $\omega$.
\bibliographystyle{plain}
\bibliography{majidian}

\begin{thebibliography}{1}

\bibitem{ash12}
A.~Asheim, A.~Deano, D.~Huybrechs, and H.~Wang.
\newblock A {G}aussian quadrature rule for oscillatory integrals on a bounded
  interval.
\newblock {\em DCDSA}, 34(3):883--901, 2014.

\bibitem{mil05}
G.~V. Milovanovi{\'c} and A.~S. Cvetkovi{\'c}.
\newblock Orthogonal polynomials and {G}aussian quadrature rules related to
  oscillatory weight functions.
\newblock {\em J. Comput. Appl. Math.}, 179(1):263--287, 2005.

\end{thebibliography}
\end{document}